\documentclass[12pt]{amsart}
\usepackage{amsfonts,amssymb,amscd,amsmath,enumerate,verbatim,calc}

\newtheorem{Theorem}{Theorem}[section]
\newtheorem{Lemma}[Theorem]{Lemma}
\newtheorem{Corollary}[Theorem]{Corollary}
\newtheorem{Proposition}[Theorem]{Proposition}
\newtheorem{Remark}[Theorem]{Remark}

\newtheorem{Example}[Theorem]{Example}

\begin{document}
\title{Resolvability in Hypergraphs}
\author{Imran Javaid$^{1}$, Azeem Haider$^2$, Muhammad Salman$^3$, Sadaf Mehtab$^1$}
%\subjclass{Primary: , Secondary: }
\keywords{metric dimension, partition dimension, hypergraph
\indent $*$\ Corresponding author: ijavaidbzu@gmail.com\\
\indent This work is the part of the thesis, written by the last author, which was submitted to Bahauddin\\
\indent Zakariya University Multan Pakistan for the fulfillment of M. Phil. degree in 2011.}
%\indent This research of the authors was partially supported by the
%Higher Education Commission of\\
%\indent Pakistan}
\address{1. Center for advanced studies in Pure and Applied Mathematics,
Bahauddin Zakariya University Multan, Pakistan.\newline \indent
E-mail: ijavaidbzu@gmail.com, mishi$_{-}$abi@hotmail.com}
\address{2. Department of Mathematics, Faculty of Science, Jazan University, Jazan, Saudi Arabia. Email: aahaider@jazanu.edu.sa}
\address{3. Department of Mathematics, The Islamia University of Bahawalpur, Punjab, Bahawalpur, 63100, Pakistan. Email: solo33@gmail.com}
\date{}
\maketitle
\begin{abstract}
This article emphasizes an extension of the study of metric
and partition dimension to hypergraphs. We give a sharp lower bounds
for the metric and partition dimension of hypergraphs in general and
give exact values under specified conditions.
\end{abstract}
%=======================================================================
%%%---------------------------------------------------------------------
\section{Introduction}
A {\it hypergraph} $H$ is a pair $(V(H),E(H))$, where $V(H)$ is a
finite non-empty set of vertices and $E(H)$ is a finite family of distinct non-empty
subsets of $V(H)$, called hyperedges, with $\bigcup \limits_{E\in
E(H)} E = V(H)$.  The ``order'' and the
``size'' of $H$ is denoted by $m$ and $k$, respectively. A {\it
subhypergraph} $K$ of a hypergraph $H$ is a hypergraph with vertex
set $V(K)\subseteq V(H)$ and edge set $E(K)\subseteq E(H)$. A hypergraph $H$ is {\it linear} if for distinct hyperedges
$E_i,E_j\in E(H)$, $|E_i\cap E_j| \leq 1$, so for a linear
hypergraph there are no repeated hyperedges of cardinality
greater than one. A hypergraph $H$ such that no hyperedge is a subset of
any other is called {\it Sperner}.\\
A vertex $v\in V(H)$ is {\it incident} with a hyperedge $E$ of $H$
if $v\in E$. If $v$ is incident with exactly $n$ hyperedges, then we
say that the {\it degree} of $v$ is $n$; if all the vertices $v \in
V(H)$ have degree $n$, then $H$ is {\it n-regular}. Similarly, if
there are exactly $n$ vertices incident with a hyperedge $E$, then
we say that the size of $E$ is $n$; if all the hyperedges $E \in
E(H)$ have size $n$, then $H$ is {\it n-uniform}. A graph is simply
a $2$-uniform hypergraph. A hyperedge $E$ of $H$ is called a {\it
pendant hyperedge} if for $E_i,E_j\in E(H)$, $E\cap E_i \neq
\emptyset$ and $E\cap E_j \neq \emptyset$ implies $(E\cap E_i)\cap
(E\cap E_j) \neq \emptyset$. A {\em path} of length $l$ from a vertex $v$ to
another vertex $u$ in a hypergraph is a finite sequence of the
form $v,E_1,w_1,E_2,w_2,...,E_{l-1}, w_{l-1},E_l,u$ such that $v\in E_1,\; w_i\in E_i\cap E_{i+1}\;
\mbox{for}\; i=1,2,...{l-1}$ and $u\in E_l.$ A hypergraph $H$ is
called {\em connected} if there is a path between any two vertices
of $H$. All hypergraphs considered in this paper are connected
Sperner hypergraphs.\\
A hypergraph $H$ is said to be a {\it hyperstar} if  there exists a
subset $C$ of vertices such that $E_i\cap E_j = C \neq \emptyset$,
for any $E_i,E_j \in E(H)$. Then $C$ is called the center of the
hyperstar. If there exists a sequence of hyperedges
$E_1,E_2,\ldots,E_k$ in a hypergraph $H$, then $H$ is said to be
(1)\ a {\it hyperpath} if $E_i\cap E_j \neq \emptyset$ if and only
if $|i-j| = 1$; (2)\ a {\it hypercycle} if, $E_i\cap E_j \neq
\emptyset$ if and only if $i-j\in \{1,-1\}$ (mod $k$).
%A {\it hyperpath} is a hypergraph $H$ with vertex set $V(H) =
%\bigcup\limits_{i=1}^{k}\{v_{1}^{i},v_{2}^{i},...,v_{n_{i}}^{i}\}$
%and an edge set $E(H) =\{E_{1},E_{2},\ldots,E_{k}\}$, where $E_{i}=
%\{v_{1}^{i},v_{2}^{i},...,v_{n_{i}}^{i} \}$ with
%$v_{n_{i}-j+l}=v_{l}^{i+1},\,\ \mbox{for}\,\ 1\leq l\leq j$ when
%$|E_{i}\cap E_{i+1}| = j$ for all $1\leq i\leq k-1$, and $E_{k}=
%\{v_{1}^{k},v_{2}^{k},...,v_{n_{i}}^{k}\}.$ A {\it hypercycle} is a
%hypergraph $H$ with vertex set $V(H)=\bigcup\limits_{i=1}^{k}
%\{v_{1}^{i},v_{2}^{i},...,v_{n_{i}}^{i}\}$ and an edge set
%$E(H)=\{E_{1},E_{2},\ldots,E_{k}\},$ where $E_{i}= \{
%v_{1}^{i},v_{2}^{i},....,v_{n_{i}}^{i}\}$  with
%$v_{n_{i}-j+l}=v_{l}^{i+1},\,\ \mbox{for}\,\ 1\leq l\leq j$ when
%$|E_{i}\cap E_{i+1}| = j$ and $i+1 $ is taken modulo $k$.
A connected hypergraph $H$ with no hypercycle is called a {\it
hypertree}. A subhypertree of a hypertree $H$ with edge set, say
$\{E_{p_1},E_{p_2},\ldots,E_{p_l}\}\subset E(H)$, is called a {\it
branch} of $H$ if $E_{p_1}$ (say) is the only hyperedge such that,
for $E_i,E_j\in E(H)\setminus\{E_{p_1},E_{p_2},\ldots,E_{p_l}\}$,
$E_{p_1}\cap E_i \neq \emptyset$ and $E_{p_1}\cap E_j \neq
\emptyset$ implies $(E_{p_1}\cap E_i)\cap (E_{p_1}\cap E_j) \neq
\emptyset$. The hyperedge $E_{p_1}$ is called the {\it joint} of the
branch.\\
%A hypergraph $H$ with $E(H) =
%\{E_1,\ldots, E_k\}$ is said to be a {\it hyperpath}, if there
%exists a permutation $\pi: \{1,\ldots,k\}\rightarrow \{1,\ldots,k\}$
%such that $(1)$\ $E_{\pi(i)}\cap E_{\pi(i+1)} \neq \emptyset$ for
%any $i = 1,\ldots,k$, $(2)$\ $E_{\pi(i)}\cap E_{\pi(j)} = \emptyset$
%if $j\not\in \{i-1,i,i+1\}$ and $2\leq i\leq k-1$. A hyperpath is
%linear if, for all $i$, $|E_{\pi(i)}\cap E_{\pi(i+1)}| = 1$. A
%hypergraph $H$ with $E(H) = \{E_1,\ldots, E_k\}$ is said to be a
%{\it hypercycle}, if there exists a permutation $\pi:
%\{1,\ldots,k\}\rightarrow \{1,\ldots,k\}$ such that $(1)$\
%$E_{\pi(i)}\cap E_{\pi(i+1\ \mbox{mod}\ k)} \neq \emptyset$ for any
%$i = 1,\ldots,k$, $(2)$\ $E_{\pi(i)}\cap E_{\pi(j)} = \emptyset$ if
%$j\not\in \{i-1\ \mbox{mod}\ k,i,i+1\ \mbox{mod}\ k\}$. A hypercycle
%is linear if, for all $i$, $|E_{\pi(i)}\cap E_{\pi(i+1\ \mbox{mod}\
%k)}| = 1$.
%For any $ a_{1},a_{2},...,a_{l} \in \mathbb{N}$ and $n\geq 3$, we
%denote by $\mathcal{H}(n;a_{1},a_{2},...,a_{l})$, the $n$-uniform
%linear hypergraph consisting of $l$, $n$-uniform linear hyperpaths
%$P_{a_l}$ of length $a_{l}$, where $l\geq 2$, joined in parallel and
%having only two end vertices in common. The hypergraph
%$\mathcal{H}(n;a_{1},a_{2},...,a_{l})$ is called a {\it multi-bridge
%hypergraph} \cite{tom}.
An ordered set $W$ of vertices of a connected graph $G$ is called a
resolving set for $G$ if for every two distinct vertices $u,v \in
V(G)$, there is a vertex $w \in W$ such that $d(u,w) \neq d(v,w)$. A
resolving set of minimum cardinality is called a basis for $G$ and
the number of vertices in a basis is called the metric dimension of
$G$, denoted by $dim(G)$. An ordered $t$-partition $\Pi$ $=$ $\{S_1,
S_2,\ldots,S_t\}$ of $V(G)$ is called a resolving partition if for
every two distinct vertices $u,v \in V(G)$, there is a set $S_i$ in
$\Pi$ such that $d(u,S_i) \neq d(v,S_i)$, where $d(v,s) = \min
\limits_{s \in S}d(u,s)$. The minimum $t$ for which
there is a resolving $t$-partition of $V(G)$ is called the partition
dimension of $G$, denoted by $pd(G)$.
In this article, we consider hypergraphs in the context of metric dimension and partition dimension,
which are defined in Sections 2 and 3, respectively. We give sharp lower bounds for the metric and
partition dimension of graphs. The metric dimension
of some well-known families of hypergraphs such as hyperpaths,
hypertrees and $n$-uniform linear hypercycles is investigated.
Further, we find the metric and partition dimension of $3$-uniform linear hypercycles.
We also characterize all the $n$-uniform (for all $n\geq 2$ and $n\neq
3$ when $k$ is even) linear hypergraphs with partition dimension
$n$. Moreover, all the hypergraphs with metric dimension 1 and
partition dimension 2 are characterized.
%======================================================================
%%%--------------------------New Section-------------------------------
\section{Metric Dimension of Hypergraphs}
%----------------------------------------------------------------------
The metric dimension of a graph was first studied by Slater
\cite{slater} and independently by Harary and Melter \cite{rp3}. It
is a parameter that has appeared in various applications, as diverse
as combinatorial optimization, pharmaceutical chemistry, robot
navigation and sonar. In recent years, a considerable literature has
been developed (see \cite{ahmad,rp2,ftr2,7,ftr1,jav,8,salman}). The
problem of determining whether $dim(H) < M$ ($M>0$), where $H$ is a
simple graph, is an NP-complete problem \cite{5,8}. The metric
dimension of a hypergraph $H$ is defined as follows:\\
The {\em distance} between any two vertices $v$ and $u$ of a hypergraph $H$, $d(v,u)$, is the length of a
shortest path between them and $d(v,u)=0$ if and only if $v=u$. The {\it diameter} of
$H$ is the maximum distance between the vertices of $H$, and is denoted
by $diam(H)$. Two vertices $u$ and $v$ of $H$ are said to be
``diametral'' vertices if $d(u,v) = diam(H)$. The {\em representation},
$r(v|W)$, of a vertex $v$ of $H$ with respect to an ordered set
$W=\{w_1,w_2,...,w_q\}\subseteq V(H)$ is the $q$-tuple
$r(v|W)=\left(d(v,w_1), d(v,w_2),...,d(v,w_q)\right).$ The set $W$
is called a {\em resolving set} for a hypergraph $H$ if $r(v|W)\neq
r(u|W)$ for any two different vertices $v,u\in V(H).$ A resolving
set with minimum cardinality is called a {\em basis} for $H$ and
that minimum cardinality is called the {\em metric dimension} of
$H$, denoted by $dim(H)$.\\
To determine whether a given set $W\subseteq V(H)$ is a resolving
set for a hypergraph $H$, $W$ needs only to be verified for the
vertices in $V(H)\setminus W$ since every vertex $w \in W$ is the
only vertex of $H$ whose distance from $w$ is $0$.\\
If we denote all the vertices of degree $d$ in $E_{i_1}\cap
E_{i_2}\cap...\cap E_{i_d}$ by the class $C{(i_1,i_2,...,i_d)},$
then the collection of all such classes gives a partition of $V(H).$
%A hypergraph $H$ is {\it vertex-simple} if
%every class $C{(i_1,i_2,...,i_d)}$ is either an empty set or a
%singleton subset of $V(H)$.
Thus, we have the following straightforward proposition:
%----------------------Proposition------------------------------
\begin{Proposition}\label{pro1}
For any two distinct vertices $u, v\in C(i_1,i_2,...,i_d)$, we have
$d(u,w)=d(v,w)$ for any $w\in V(H)\setminus\{u,v\}.$
\end{Proposition}
Thus, we extract the following Lemma related to the resolving set for
$H$:
\begin{Lemma}\label{l1}
If $u,v\in C(i_1,i_2,...,i_d)$ and $W\subseteq V(H)$ resolves $H$,
then at least one of the vertices $u$ and $v$ is in $W$. Moreover,
if $u\in W$ and $v\not\in W$, then $(W\setminus\{u\})\cup \{v\}$
also resolves $H$.
\end{Lemma}
Let us denote $n(i_1,i_2,...,i_d)=|C{(i_1,i_2,...,i_d)}|-1$ when
$C{(i_1,i_2,...,i_d)}\neq \emptyset$, otherwise we take
$n(i_1,i_2,...,i_d)=0.$ This notation helps us to write a
lower bound for the metric dimension of hypergraphs in the following Proposition.
%---------------------Proposition 2-------------------------------------------------
\begin{Proposition}\label{p1}
For any hypergraph $H$ with $k$ hyperedges, $$dim(H)\geq
\sum\limits_{j=1}^{k}\sum\limits_{i_1<..<i_j}^{k}n(i_1,i_2,...,i_j).$$
\end{Proposition}
\begin{proof}
It follows from the fact that if there are $|C(i_1,i_2,...,i_d)|$
number of vertices of degree $d$ in $E_{i_1}\cap E_{i_2}\cap...\cap
E_{i_d},$ then, by Lemma \ref{l1}, at least $n(i_1,i_2,...,i_d)$
vertices should belong to any basis $W$.
\end{proof}
%-------------------Remark-----------------------------------------------------------
\begin{Remark}\label{r1}
By Proposition \ref{p1}, it is clear that, in order to obtain a
basis of any hypergraph $H$, it suffices to consider only one
vertex, say $v_{i_1,i_2,...,i_d}$, from each class
$C(i_1,i_2,...,i_d)$ if $C(i_1,i_2,...,i_d)\neq \emptyset$. We call
this vertex, a representative vertex of the class $C(i_1,i_2,...,i_d).$ We
denote the set of all representative vertices in a hypergraph $H$ by
$R(H)$, and hence we always have, $V(H)\setminus R(H)\subseteq W$ for
any basis $W$ of $H$.
\end{Remark}
Now we discuss some classes of hypergraphs for which the equality holds in
the Proposition \ref{p1}.
%---------------------Theorem 1------------------------------------------------------
\begin{Theorem}\label{t1}
For any hypergraph $H$ with $k$ hyperedges, if $n(i)\neq 0$ for all
$E_i\in E(H),$ then $dim(H)=
\sum\limits_{j=1}^{k}\sum\limits_{i_1<..<i_j}^{k}n(i_1,i_2,...,i_j).$
Moreover, there are
$\prod\limits_{j=1}^{k}\prod\limits_{i_1<..<i_j}^{k} (n(i_1,i_2,\\
...,i_j)+1)$ basis for $H$.
\end{Theorem}

\begin{proof}
Consider $W=V(H)\setminus R(H),$ we have to show that $W$ is a basis
for $H$. Take any two different vertices $v,v'\in R(H)$. Since both
the vertices $v$ and $v'$ are representative vertices of different
classes, there exists a hyperedge $E_j$ such that $v'\in E_j$ and
$v\not\in E_j.$ It follows from $n(j)\neq 0$ that there exists a
vertex of degree one $w_j \in V(H)$ such that $w_j\in E_j\cap W$.
Clearly, $d(v',w_j)=1$ and $d(v,w_j)\neq 1,$ hence $W$ is a basis
for $H$. Further, by Lemma \ref{l1}, there are
$\prod\limits_{j=1}^{k}\prod\limits_{i_1<..<i_j}^{k}
(n(i_1,i_2,...,i_j)+1)$ such $W$.
\end{proof}

For all $n\geq 4$, if $H$ is an $n$-uniform linear hypergraph with
$k$ hyperedges, then $n(i,i+1) = 0$ for every $i\in \{1,2,\ldots,k\}$. Thus, we have the
following corollary:

\begin{Corollary}\label{c11}
For $n\geq 4$, let $H$ be an $n$-uniform linear hypergraph with $k$
hyperedges. If $n(i)\neq 0$ for all $E_i\in E(H),$ then $dim(H) =
\sum \limits_{i = 1}^k n(i)$.
\end{Corollary}

We give two examples which show that the condition in Theorem
\ref{t1} cannot be relaxed generally.

\begin{Example}\label{e1}
Let $H$ be a hypergraph with vertex set $V(H)=\{v_1,v_2,v_3,v_4\}$
and edge set $E(H)=\{E_1, E_2\}$, where $E_1=\{v_1,v_2,v_3\}$ and
$E_2=\{v_3,v_4\}$. Clearly, $n(2)=0$ so $H$ does not satisfy the
condition of Theorem \ref{t1}. Without loss of generality, we can
take the set of representative vertices $R(H)=\{v_1,v_3,v_4\}$, and
hence $W=V(H)\setminus R(H) = \{v_2\}$. But, $W$ is not a resolving
set for $H$ since $r(v_1|W)=r(v_3|W).$ In fact, $dim(H)=2>1.$
\end{Example}

%Now, we give an example which shows the, if $n(i) = 0$ and $n(i,i+1)
%\neq 0$, then still the result in Theorem \ref{t1} does not work.

\begin{Example}\label{e11}
Let $H$ be a hypergraph with vertex set
$V(H)=\{v_1,v_2,v_3,v_4,v_5,v_6\}$ and edge set
$E(H)=\{E_1,E_2,E_3\}$, where
$E_1=\{v_1,v_2,v_3,v_4\},E_2=\{v_3,v_4,v_5,v_6\}$ and
$E_3=\{v_1,v_2,v_5,v_6\}$. Clearly, $n(i) = 0$ for all $i = 1,2,3$
and $n(1,2)= n(2,3) = n(3,1) \neq 0$. Without loss of generality, we
can take the set of representative vertices $R(H)=\{v_1,v_3,v_5\}$,
and hence $W=V(H)\setminus R(H) = \{v_2,v_4,v_6\}$. But, $W$ is not
a resolving set for $H$ since $r(v_1|W)=r(v_3|W) = r(v_5|W).$ In
fact, $dim(H)=5>3.$
\end{Example}

However, the condition in Theorem \ref{t1} can be reduced in some
special cases as shown in the following results.

\begin{Theorem}\label{l3}
Let $H$ be a hyperpath with $k$ hyperedges $E_1,E_2,\ldots E_k$ in a
canonical way. Then
$dim(H)=\sum\limits_{i=1}^{k}n(i)+\sum\limits_{i=1}^{k-1}n(i,i+1)$
if both $n(1)$ and $n(k)$ are non-zero.
\end{Theorem}

\begin{proof}
Let $W = V(H)\setminus R(H)$. Then it follows from the facts
$n(1)\neq 0$ and $n(k)\neq 0$ that there exists a vertex of degree
one $w_1\in E_1\cap W$ and there exists a vertex of degree one
$w_k\in E_k\cap W.$  In order to prove the theorem, we only have to
show that the representative vertices are resolved by the set $W$,
and it yields from the fact that for any $1\leq j\leq k$, we have
$\left(d(v_j,w_1), d(v_j,w_k)\right)=(j,k-j+1)$, and for any $1\leq
j<k-1$, we have $\left(d(v_{j,j+1},
w_1),d(v_{j,j+1},w_k)\right)=(j,k-j).$
\end{proof}

\begin{Theorem}\label{t2}
Let $H$ be a hypertree with $k$ hyperedges and let $E_{p1},
E_{p2},\dots,E_{pt}$ be its pendant hyperedges. Then
$dim(H)=\sum\limits_{j=1}^{k}\sum\limits_{i_1<..<i_j}^{k}n(i_1,i_2,...,i_j)$
if $n(ps)\neq 0$ for all $s=1,2,\ldots,t.$
\end{Theorem}

\begin{proof}
Consider $W=V(H)\setminus R(H),$ similarly as in the proof of
Theorem \ref{t1}, again we have to show that $W$ is a basis for $H$.
Take any two different vertices $v,v'\in R(H)$, then both vertices
are representative of two different classes, and hence there exists
a hyperedge $E_j$ such that $v'\in E_j$ but $v\not\in E_j.$ Now,
consider a hyperpath contained in the hypertree $H$ which starts and
ends at the pendant hyperedges and contains both $v$ and $E_j.$ By
using the proof of Theorem \ref{l3}, it can be seen that the
vertices $v$ and $v'$ has different representations with respect to
$W$, which proves the theorem.
\end{proof}

An $n$-uniform linear hyperstar $(n \geq 3)$ is a special case of hypertree in which $n(i)\neq 0$ for all $E_i\in E(H)$, so we
have the following corollary:

\begin{Corollary}\label{c22}
For $n\geq 3$, let $H$ be an $n$-uniform linear hyperstar with $k\
(\geq 3)$ hyperedges. Then $dim(H)= k(n-2)$.
\end{Corollary}

Consider an $n$-uniform linear hypercycle $\mathcal{C}_{k,n}$ with
$k$ hyperedges. When $n\geq 4$, then $n(i)\neq 0$ for all $E_i\in
E(\mathcal{C}_{k,n})$ so, by Corollary \ref{c11},
$dim(\mathcal{C}_{k,n}) = k(n-3)$.

%\begin{Theorem}\label{tt2}
%Let $\mathcal{C}_{k,n}$ be an $n$-uniform hypercycle with $k$
%hyperedges and $n\geq 4$. Then
% $dim(\mathcal{C}_{k,n})=\sum\limits_{i=1}^{k}n(i)
%+\sum\limits_{i=1}^{k-1}n(i,i+1)=\sum\limits_{i=1}^{k}n(i)=k(n-3).$
%\end{Theorem}

For the case $n=3$, we have $n(i) = 0$ for all $E_i\in E(H)$, hence
the lower bound given in Proposition \ref{p1} is zero and every
vertex in $\mathcal{C}_{k,3}$ is the representative vertex. We
discuss this case in the following result:

\begin{Theorem}\label{t3}
Let $\mathcal{C}_{k,3}$ be a $3$-uniform linear hypercycle with $k$
hyperedges. Then $dim\left(\mathcal{C}_{3,3}\right)$ $=2$ and for
all $k\geq 4$,
$$ dim\left(\mathcal{C}_{k,3}\right) = \left\{
            \begin{array}{ll}
             2,          \,\,\,\     & \mbox{if}\,\,\,\ k \mbox{ is even},\\
             3,          \,\,\,\     & \mbox{if}\,\,\,\ k \mbox{ is odd}.
            \end{array}
             \right.
$$
\end{Theorem}

\begin{proof}
In $\mathcal{C}_{k,3}$, each $v_j\in E_j$ represents a vertex of
degree one and $v_{j,j+1}\in E_j\cap E_{j+1}$ with
$v_{k,k+1}=v_{k,1}$. Clearly, $dim(\mathcal{C}_{k,3})>1$ for any
$k.$

\indent If $k$ is even, then we take $W=\{v_1, v_{\frac{k}{2}}\}$.\\
For $1< j< \frac{k}{2},$ we have $r(v_j|W)=(j,\frac{k}{2}-j+1)$ and
for $1\leq j < \frac{k}{2},\;r(v_{j,j+1}|W)=(j,\frac{k}{2}-j).$ Now,
if $\frac{k}{2}+1\leq j< k,$ then $r(v_j|W)=(k+2-j,j-\frac{k}{2}+1)$
and $r(v_{j,j+1}|W)=(k+1-j,j-\frac{k}{2}+1)$ with
$r(v_k|W)=(2,\frac{k}{2}+1)$,
$r(v_{\frac{k}{2},\frac{k}{2}+1}|W)=(\frac{k}{2},1)$ and
$r(v_{k,1}|W)=(1,\frac{k}{2}).$ It is easy to see that the
representations of all the vertices with respect to $W$ are
distinct, hence $W$ forms a basis for $\mathcal{C}_{k,3}$ and
$dim(\mathcal{C}_{k,3})=2$.

\indent For the special case when $k=3$, the set $W=\{v_1,v_2\}$
forms a basis for $\mathcal{C}_{3,3}$. Hence
$dim(\mathcal{C}_{3,3})=2.$

\indent If $k>3$ is odd, then we first show that
$dim(\mathcal{C}_{k,3})>2.$ Suppose on contrary that
$dim(\mathcal{C}_{k,3}) = 2$ and let $W$ is a basis of
$\mathcal{C}_{k,3}$. Let us call the vertices $v_{i,i+1}$, $i\in
\{1,2,\ldots, k\}$, of $\mathcal{C}_{k,3}$, the common vertices. We
have the following three possibilities:

$(1)$\ $W$ contains both common vertices. Without loss of
generality, we may assume that one vertex is $v_{1,2}$ and the second
vertex is $v_{j,j+1}$ $(2\leq j\leq k)$. Then $r(v_{j+1}|W) =
r(v_{j+1,j+2}|W)$, for $2\leq j < \frac{k+1}{2}$; $r(v_2|W) =
r(v_{k,1}|W)$, for $j = \frac{k+1}{2}$; $r(v_1|W) = r(v_{2,3}|W)$,
for $j = \frac{k+1}{2}+1$ and $r(v_j|W) = r(v_{j-1,j}|W)$, for
$\frac{k+1}{2}+1 < j\leq k$, a contradiction.

$(2)$\ $W$ contains one common vertex. Without loss of generality,
we may assume that one vertex is $v_{1,2}$ and the second vertex is
$v_{j}$ $(1\leq j\leq k)$. Then $r(v_{j+1}|W) = r(v_{j+1,j+2}|W)$,
for $1\leq j < \frac{k+1}{2}$; $r(v_1|W) = r(v_{k,1}|W)$, for $j =
\frac{k+1}{2}$; $r(v_1|W) = r(v_{2}|W)$, for $j = \frac{k+1}{2}+1$
and $r(v_2|W) = r(v_{2,3}|W)$, for $\frac{k+1}{2}+1 < j\leq k$, a
contradiction.

$(3)$\ $W$ contains no common vertex. Without loss of generality, we
may assume that one vertex is $v_{1}$ and the second vertex is $v_{j}$
$(2\leq j\leq k)$. Then $r(v_{j+1}|W) = r(v_{j+1,j+2}|W)$, for
$2\leq j < \frac{k+1}{2}$; $r(v_{j-1}|W) = r(v_{j+1,j+2}|W)$, for $j
= \frac{k+1}{2}$; $r(v_{j+1}|W) = r(v_{j-2,j-1}|W)$, for $j =
\frac{k+1}{2}+1$ and $r(v_{j-1}|W) = r(v_{j-2,j-1}|W)$, for
$\frac{k+1}{2}+1 < j\leq k$, a contradiction.

\indent Now, we will show that $dim(\mathcal{C}_{k,3})\leq 3$. Take
$W=\{v_1,v_2,v_{\frac{k+1}{2}}\}$. We note that, $r(v_{1,2}|W) =
(1,1,\frac{k-1}{2})$ and
$$\,\,\,\,\,\,\,\,\,\,\,\,\,\,\,\,\ r(v_j|W) = \left\{
            \begin{array}{ll}
            (j,j-1,\frac{k+1}{2}-j+1)        &\,\ \mbox{for}\,\ 2<j<\frac{k+1}{2},\\
            (\frac{k+1}{2},\frac{k+1}{2},2)  &\,\ \mbox{for}\,\ j = \frac{k+1}{2}+1,\\
            (k-j+2, k-j+3,j-\frac{k-1}{2})   &\,\ \mbox{for}\,\ \frac{k+1}{2}+1< j \leq k,
            \end{array}
             \right.
$$

$$ r(v_{j,j+1}|W) = \left\{
            \begin{array}{ll}
            (j,j-1,\frac{k+1}{2}-i)             &\,\ \mbox{for}\,\ 2\leq j <\frac{k+1}{2},\\
            (\frac{k+1}{2},\frac{k-1}{2},1)     &\,\ \mbox{for}\,\ j = \frac{k+1}{2},\\
            (k-j+1, k-j+2, j-\frac{k-1}{2})     &\,\ \mbox{for}\,\ \frac{k+1}{2}< j \leq k.
            \end{array}
             \right.
$$

One can see that all the vertices of $V(\mathcal{C}_{k,3})-
W$ have distinct representations. This implies that
$dim(\mathcal{C}_{k,3})= 3$ when $k > 3$ is odd.
\end{proof}

The {primal graph}, $prim(H)$, of a hypergraph $H$ is a graph with
vertex set $V(H)$ and vertices $x$ and $y$ of $prim(H)$ are adjacent
if and only if $x$ and $y$ are contained in a hyperedge. The {\it
middle graph}, $M(H)$, of $H$ is a subgraph of $prim(H)$ obtained by
deleting loops and parallel edges. Since the adjacencies between the vertices in $prim(H)$ are due to the adjacencies in the hypergraph $H$, so determining the length of a path between two vertices $u$ and $v$ in $prim(H)$ is equivalent to determine the length of a path between the vertices $u$ and $v$ in $H$. This fact yields the following result:

\begin{Theorem}\label{tt7}
Let $H$ be a hypergraph. Then $$dim(H) = dim(prim(H)) = dim(M(H)).$$
\end{Theorem}

%The primal graph of a hypergraph $H$ is a simple graph (without
%loops and parallel edges), which is also the middle graph. But,
The $dual$ of $H = (\{v_1,v_2,\ldots, v_m\},
\{E_1,E_2,\ldots,E_k\})$, denoted by $H^*$, is the hypergraph whose
vertices are $\{E_1,E_2,\ldots, E_k\}$ corresponding to the
hyperedges of $H$ and with hyperedges $V_i = \{E_j\ :\ v_i\in E_j\
\mbox{in}\ H\}$, where $i = 1 ,2,\ldots, m$. In other words, the
dual $H^*$ swaps the vertices and hyperedges of $H$. The primal
graph of the dual $H^*$ of a hypergraph $H$ is not a simple graph,
in this case, the middle graph of $H^*$ is a simple graph. We
discuss the metric dimension of dual hypergraphs separately in the
following result, which also helps us to characterize all the
hypergraphs with metric dimension one.

\begin{Theorem}\label{tt5}
Let $H^*$ be the dual of a hypergraph $H$. Then $$dim(H^*) = dim(M(H^*)).$$
\end{Theorem}

\begin{proof}
By the definition of middle graph, for any two vertices $u$ and $v$
of $H^*$, a path $P$ is a shortest path between the vertices $u$ and
$v$ in $H^*$ if and only if $P$ is a shortest path between $u$ and
$v$ in $M(H^*)$. Thus a set $W\subseteq V(H^*)$ is a minimum
resolving set for $H^*$ if and only if $W$ is a minimum resolving
set for $M(H^*)$.
\end{proof}

The middle graph of $H^*$ is (1)\ a simple path $P_m$ if and only if
$H$ is a hyperpath; (2)\ a simple cycle $C_m$ if and only if $H$ is
a hypercycle. In \cite{rp2}, all the simple connected graphs having
metric dimension one were characterized by proving the result
``$dim(G)$ is one if and only if $G$ is a simple path $P_m$ ($m\geq
1$)''. Now, we characterize all the connected hypergraphs having the
metric dimension 1. In fact, all these hypergraphs are the dual
hypergraphs and have been characterized in the following consequence
of Theorem \ref{tt5}.

\begin{Corollary}\label{c1}
Let $H^*$ be the dual of a hypergraph $H$. Then $dim(H^*)
= 1$ if and only if $H$ is a hyperpath.
\end{Corollary}

In \cite{rp3}, it was shown that the metric dimension of a simple
cycle $C_m$ ($m\geq 3$) is two. % and $dim(T) = \sigma(T) - ex(T)$, where
%$\sigma(T)$ denotes the sum of the terminal degrees of the major
%vertices of $T$ and $ex(T)$ denotes the number of exterior major
%vertices of $T$ \cite{rp2}.
Thus, we have the following corollary:

\begin{Corollary}\label{c2}
Let $H^*$ be the dual of a hypercycle $H$. Then $dim(H^*)
= 2$.
\end{Corollary}

%\begin{Corollary}\label{c3}
%Let $H$ be a $p$-hypertree and $H^*$ be the dual of $H$. Then
%$dim(H^*) = \sigma(M(H^*)) - ex(M(H^*))$.
%\end{Corollary}

%%%--------------------------------------------------------------------
\section{Partition Dimension of Hypergraphs}
Possibly to gain insight into the metric dimension,
Chartrand $et$ $al$. introduced the notion of a resolving partition
and partition dimension \cite{rp1,cha2}. To define the partition
dimension, the distance $d(v,S)$ between a vertex $v$ in $H$ and
$S\subseteq V(H)$ is defined as $\min \limits_{s\in S}d(v,s).$ Let
$\Pi=\{S_1, S_2,\ldots, S_t\}$ be an ordered $t$-partition of $V(H)$
and $v$ be any vertex of $H.$ Then the {\em representation},
$r(v|\Pi)$, of $v$ with respect $\Pi$ is the $t$-tuple
$r(v|\Pi)=\left(d(v,S_1), d(v,S_2),...,d(v,S_t)\right).$ The
partition $\Pi$ is called a {\em resolving partition} for a
hypergraph if $r(v|\Pi)\neq r(u|\Pi)$ for any two distinct vertices
$v,u\in V(H)$. The {\em partition dimension} of a hypergraph $H$ is
the cardinality of a minimum resolving partition, denoted by
$pd(H)$.

From the definition of a resolving partition, it can be observed
that the property of a given partition $\Pi$ of a hypergraph $H$ to
be a resolving partition of $H$ can be verified by investigating the
pairs of vertices in the same class. Indeed, $d(x,S_i)=0$ for every
vertex $x \in S_i$ but $d(x,S_j)\neq 0$ with $j \neq i$. It follows
that $x \in S_i$ and $y \in S_j$ are resolved either by $S_i$ or
$S_j$ for every $i\neq j$.
%Chartrand $et$ $al$. \cite{rp1} gave the relationship between metric
%and partition dimension of a connected graph $G$ by showing that
%$pd(G)\leq dim(G)+1$. Since a hypergraph $H$ is a generalization of
%a graph $G$, we deduce that $pd(H)\leq dim(H)+1$.
From Proposition \ref{pro1}, we have the following lemma:

\begin{Lemma}\label{l4}
Let $\Pi$ be a resolving partition of $V(H)$. If $u,v \in
C(i_1,i_2,...,i_d)$ then $u$ and $v$ belong to distinct classes of
$\Pi$.
\end{Lemma}

The following result gives the lower bound for the partition
dimension of hypergraphs.

\begin{Proposition}\label{p2}
Let $H$ be a hypergraph with $k$ hyperedges. Then $pd(H)\geq \lambda
+1$, where $\lambda = \max |C(i_1,i_2,\ldots, i_d)|$ in $H$.
\end{Proposition}

\begin{proof}
Since $\lambda = \max |C(i_1,i_2,\ldots, i_d)|$ in $H$, by Lemma
\ref{l4}, we have at least $\lambda$ disjoint classes $S_1,
S_2,\ldots, S_{\lambda}$ of $V(H)$. Since $H$ is Sperner so there
exists an edge $E$ of $H$ such that $C(i_1,i_2,\ldots, i_d)\subset
E$. Now, if $\Pi
 = \{S_1,\ldots,S_{\lambda}\}$ is a minimum resolving partition of $V(H)$ then there exist two vertices $u$ and $v$ in
 $E$ such that $u,v \in S_i$ (say) with $r(u|\Pi) = (1,\ldots,0,\ldots,1) = r(v|\Pi)$, where $0$ is at the
$i$th place, a contradiction. Thus, $pd(H)\geq \lambda +1$.
\end{proof}

The lower bound given in Proposition \ref{p2} is sharp for an
$n$-uniform linear hyperpath.

%In \cite{cha2}, it was shown that the partition dimension of a
%2-uniform linear hypercycle (simple cycle) is three. The following
%result shows that the partition dimension of an $n$-uniform linear
%hypercycle is $n+1$ except when the number of edges is even in
%3-uniform linear hypercycle.
A 2-uniform hypercycle $\mathcal{C}_{k,2}$ is a simple connected cycle on
$m$ vertices and it was shown that the partition dimension of a
simple connected cycle is 3 \cite{cha2}, so $pd(\mathcal{C}_{k,2}) = 3$. In the next result, we investigate the partition dimension of 3-uniform hypercycle $\mathcal{C}_{k,3},\ k\geq 3$.

\begin{Theorem}\label{tt8}
Let $\mathcal{C}_{k,3}$ be a $3$-uniform linear
hypercycle with $k\geq 3$ hyperedges. Then, $pd(\mathcal{C}_{k,3}) = 3$.
%and
%$$ pd\left(\mathcal{C}_{k,3}\right) = \left\{
%            \begin{array}{ll}
%             3,          \,\,\,\     & \mbox{if}\,\,\,\ k \mbox{ is even},\\
%             4,          \,\,\,\     & \mbox{if}\,\,\,\ k \mbox{ is odd}.\\
%            \end{array}
%             \right.
%$$
\end{Theorem}

\begin{proof}
For all $k\geq 3$, we denote the vertices of $\mathcal{C}_{k,3}$ by
$v_i^j$, where $j$ ($1 \leq j\leq k$) represents the hyperedge
number of $\mathcal{C}_{k,3}$ and $i$ ($1\leq i\leq 3$) represents
the vertex number of the $j$th hyperedge. Each $v_2^j\in E_j$ represents the vertex of degree one and $v_3^{j} =
v_1^{j+1} \in E_j\cap E_{j+1}$ represents a vertex of degree 2 with
$v_3^{k} = v_1^{1}$.

If we put all the vertices of $\mathcal{C}_{k,3}$ into two classes
$S_1$ and $S_2$, then they do not form a resolving partition $\Pi$
of $V(\mathcal{C}_{k,3})$, because one can easily check that there
exist two vertices $u,v$ of $\mathcal{C}_{k,3}$ in a class such that
$r(u|\Pi) = (0,1)= r(v|\Pi)$. Thus $pd(\mathcal{C}_{k,3}) \geq 3$.
On the other hand, $pd(\mathcal{C}_{k,3}) \leq 3$, because we have a
resolving partition of cardinality 3 for $pd(\mathcal{C}_{k,3})$ in
each of the following case:

\noindent For $k\equiv 0$\ (mod 6), we have a resolving partition for $pd(\mathcal{C}_{k,3})$ as
$$\Pi= \{\{v_{1}^{1},\ldots,v_{3}^{\frac{1}{3}k}\}, \{v_{2}^{\frac{1}{3}k+1},\ldots,v_{3}^{\frac{2}{3}k}\},\{v_{2}^{\frac{2}{3}k+1},\ldots,v_{2}^{k}\}\}.$$

\noindent For $k\equiv 1,4$\ (mod 6), we have a resolving partition for $pd(\mathcal{C}_{k,3})$ as
$$\Pi= \{\{v_{1}^{1},\ldots,v_{2}^{\frac{1}{3}(k+2)}\}, \{v_{3}^{\frac{1}{3}(k+2)},\ldots,v_{3}^{\frac{2}{3}(k+2)-1}\},\{v_{2}^{\frac{2}{3}(k+2)},\ldots,v_{2}^{k}\}\}.$$

\noindent For $k\equiv 2$\ (mod 6), we have a resolving partition for $pd(\mathcal{C}_{k,3})$ as
$$\Pi= \{\{v_{1}^{1},\ldots,v_{2}^{\frac{1}{3}(k+1)}\}, \{v_{3}^{\frac{1}{3}(k+1)},\ldots,v_{3}^{\frac{2}{3}(k+1)-1}\},\{v_{2}^{\frac{2}{3}(k+1)},\ldots,v_{2}^{k}\}\}.$$

\noindent For $k\equiv 3$\ (mod 6), we have a resolving partition for $pd(\mathcal{C}_{k,3})$ as
$$\Pi= \{\{v_{2}^{1},\ldots,v_{2}^{\frac{1}{3}k+1}\}, \{v_{3}^{\frac{1}{3}k+1},\ldots,v_{2}^{\frac{2}{3}k+1}\},\{v_{3}^{\frac{2}{3}k+1},\ldots,v_{3}^{k}\}\}.$$

\noindent For $k\equiv 5$\ (mod 6), we have a resolving partition for $pd(\mathcal{C}_{k,3})$ as
$$\Pi= \{\{v_{2}^{1},\ldots,v_{3}^{\frac{1}{3}(k+1)}\}, \{v_{2}^{\frac{1}{3}(k+1)+1},\ldots,v_{2}^{\frac{2}{3}(k+1)}\},\{v_{3}^{\frac{2}{3}(k+1)},\ldots,v_{3}^{k}\}\}.$$

\end{proof}

In \cite{rp1}, it was shown that a 2-uniform linear hyperpath
(simple path) has partition dimension 2. Now, we generalize this
result by proving that if $H$ is an $n$-uniform linear hyperpath
$(n\geq 2)$, then the partition dimension of $H$ is $n$.

\begin{Theorem}\label{tt7}
For $n \geq 2$, let $H$ be an $n$-uniform linear hypergraph with $k$
hyperedges. Then, for a 3-uniform linear hyperpath $H$ with even
hyperedges, $pd(H) = 3$ and for all other values of $n$, $pd(H) = n$
if and only if $H$ is a hyperpath.
\end{Theorem}

\begin{proof} Let $H$ be an $n$-uniform linear hyperpath. Then it
is a routine exercise to verify that a partition $\Pi =
\{S_1,S_2,\ldots, S_n\}$ of $V(H)$, where each $S_i,\ 1\leq i\leq
n-1$, contains the $i$th vertex of every hyperedge of $H$ and $S_n$
contains the $n$th vertex of the $k$th hyperedge, is a minimum
resolving partition.

Conversely, suppose that $\Pi = \{S_1,S_2,\ldots,S_n\}$ be a minimum
resolving partition of $V(H)$ and $H$ is an $n$-uniform linear
hypergraph. For $n = 2$, $H$ is a $2$-uniform linear hyperpath since
the partition dimension of a graph is 2 if and only if the graph is
a simple path (2-uniform linear hyperpath) \cite{rp1}. For $n = 3$,
$k$, odd and for all $n\geq 4$, if $H$ is not a hyperpath then
either $H$ contains a hypercycle or $H$ is a hypertree. Suppose that
$H$ contains a hypercycle, then by using the similar arguments as
given in the proof of Theorem \ref{tt8}, we can see that $pd(H)\geq
n+1$, a contradiction. Now, suppose that $H$ is a hypertree.
Consider a path $P:v,E_1,w_1,E_2,w_2,...,E_{l-1}, w_{l-1},E_l,u$
between two diametral vertices $v$ and $u$ in $H$. Then $P$ contains
either a pendant hyperedge, say $E_p$, or a branch with joint
$E_{p_1}$ (say), or both a pendant hyperedge and a branch. In the
first case,
%with three
%pendant hyperedges, say $E_{p_1}, E_{p_2}$ and $E_{p_3}$(Same
%argument can be extended to any number of pendant hyperedges). We
%suppose that $E_{p_1} = E_1$ and $E_{p_2} = E_k$.
if $|E_{p}\cap (E_i\cap E_j)| = 1$ ($i\neq j$), then there exist two
vertices $x,y$ in $H$, either $x\in E_{p}$ and $y\in E_i$ or $E_j$,
or $x\in E_{i}$ and $y\in E_j$, such that $x,y \in S_t$ (say) and
have $r(x|\Pi) = (1,\ldots,0,\ldots,1) = r(y|\Pi)$, where $0$ is at
the $t$th place. If $|E_{p}\cap E_i| = 1$ for all $i\neq 1, l$, then
there are two vertices $x\in E_{p}$ and $y\in E_i$ such that $x,y
\in S_j$ (say) and have $r(x|\Pi) = (1,\ldots,0,\ldots,1) =
r(y|\Pi)$, where $0$ is at the $j$th place, a contradiction to the
fact that $\Pi$ is resolving partition. Similarly, in the second and
third case, we can see that a partition of cardinality $n$ is not a
resolving partition of $V(H)$. Thus $H$ is an $n$-uniform linear
hyperpath.
\end{proof}

The {\it rank} of a hypergraph $H$ is the maximum number of
vertices in a hyperedge. One might think that the partition
dimension of $H$ is always greater than or equal to the rank of $H$.
This is true for an $n$-uniform linear hyperpath and an $n$-uniform
linear hypercycle $\mathcal{C}_{k,3}$. But, in general, it is not true as shown in the following example:

\begin{Example}\label{e2}
Let $H$ be a hypergraph with vertex set $V(H)=\{v_i:\ 1\leq i\leq
11\}$ and edge set $E(H)=\{E_1, E_2\}$, where $E_1=\{v_i;\ 1\leq i\leq
7\}$ and $E_2=\{v_i;\ 6\leq i\leq 11\}$. Clearly, $rank(H)= 7$,
$\lambda = 5$ and $\Pi = \left\{S_i = \{v_i,v_{i+5}\};\ 1\leq i\leq
5, S_6 = \{v_{11}\}\right\}$ is a minimum resolving partition of
V(H). This implies that $pd(H)  = 6 \neq rank(H)$.
\end{Example}

Likewise the results on the metric dimension of the primal and dual
graph of a hypergraph, we have the following two results on the
partition dimension of the primal and dual graph of a hypergraph,
respectively:

\begin{Theorem}\label{ttt7}
Let $H$ be a hypergraph. Then $pd(H) = pd(prim(H))$.
\end{Theorem}

\begin{Theorem}\label{tt6}
Let $H^*$ be the dual of a hypergraph $H$ and $M(H^*)$ be the middle
graph of $H^*$. Then $pd(H^*) = pd(M(H^*))$.
\end{Theorem}

Since, it was shown that the simple paths $P_m$ are the only graphs
with $pd(P_m) = 2$ \cite{rp1} and the partition dimension of the
simple cycles $C_m$ is 3, so, by Theorem \ref{tt6}, we have the
following corollaries:

\begin{Corollary}\label{c4}
Let $H^*$ be the dual of a hypergraph $H$. Then $pd(H^*)
= 2$ if and only if $H$ is a hyperpath.
\end{Corollary}

\begin{Corollary}\label{c5}
Let $H^*$ be the dual of a hypercycle $H$. Then $pd(H^*) = 3$.
\end{Corollary}

%From the Corollaries \ref{c1}, \ref{c4} and the Corollaries
%\ref{c2}, \ref{c5}, we note that $pd(H^*) \leq dim(H^*)+1$. But the
%relation, $pd(H)\leq dim(H)+1$ between the metric dimension and the
%partition dimension of any hypergraph $H$, is not true, in general.
%Because, There exists  a family of hypergraphs $H$ such that $pd(H)>
%dim(H)+1$, as shown in the following example:

%\begin{Example}\label{e2}
%Consider a hypergraph $H$ with even $k\ (\geq 4)$ edges $E_1 =
%E_{k+1}, E_2,\ldots,\\ E_k$, where $|E_i| = 3,\ |E_i\cup E_{i-1}| =
%1$ for all $i = 1, 2, \ldots, k$ and $|E_i\cup E_j|= \emptyset$ if
%$|i-i| \neq 1$. The vertex set of $H$ is $V(H) = \{v_i, v_{i,i+1}:\
%1\leq i\leq k\}$, where $v_i\in E_i$ is a vertex of degree one,
%$v_{i,i+1}\in E_i\cap E_{i+1}$ and $v_{k,k+1} = v_{k,1}$. Then
%$dim(H) = 2$ (see Theorem \ref{t3}) and $pd(H) = 4$. Because, one
%can easily check that a partition $\Pi$ with four classes $S_1 =
%\{v_i,v_{k,1}:\ 2\leq i\leq k-1\}$, $S_2 = \{v_{i,i+1},v_{1}:\ 2\leq
%i\leq k-1\}$, $S_3 = \{v_{1,2}\}$ and $S_4 = \{v_{k}\}$, is a
%minimum resolving partition of $V(H)$.
%\end{Example}

%Thus, we leave to the reader the following open problems:\\

%\noindent {\bf Open Problem 3.8.}\ {\it For every pair of integers
%$a,b$ with $a\leq b+1$, is there exists a connected hypergraph $H$
%such that $pd(H) = a$ and $dim(H) = b$?}

\end{document}